\documentclass[10pt]{article}

\usepackage{amsmath,amssymb,amsthm,mathtools}
\usepackage[margin=1in]{geometry}

\input{tcilatex}

\title{\textbf{On Standard perturbations of the Affine twist map}}

\author{Salvador Addas-Zanata\thanks{ Salvador would like to thank FAPESP (project 2023/07076-4) for partial funding.}}

\date{}

\begin{document}

\maketitle

\centerline { {\sl Instituto de Matem\'atica e Estat\'\i stica }}

\centerline {{\sl Universidade de S\~ao Paulo}}

\centerline {{\sl Rua do Mat\~ao 1010, Cidade Universit\'aria,}} 

\centerline {{\sl 05508-090 S\~ao Paulo, SP, Brazil}}

\begin{abstract}
For any $t\in\mathbb{R}$, consider the Affine twist map $\rm{Aff_t}:\mathbb{T}^2\to\mathbb{T}^2$ given by $$\rm{Aff_t}(x,y)=(x+y \text{ mod 1}, y+t \text{ mod 1}).$$ The map $\rm{Aff_t}$ clearly possesses an invariant foliation by horizontal curves. If $t$ is rational, each leaf is periodic, and if $t$ is irrational, the orbit of each leaf is dense in the torus. In both cases, the vertical rotation set of the proper lift of $\rm{Aff_t}$ to the vertical cylinder is reduced to $\{t\}.$ Now, for $k\in\mathbb{R},$ define $$f_{k,t}(x,y):=(x+y+k\sin (2\pi x)\text{ mod 1},y+k\sin (2\pi x)+t\text{ mod 1}).$$ 
We show that:
\begin{enumerate}
\item when $t=p/q$ for integers $p$ and $q>0$, KAM theory implies the existence of a constant $k_{p/q}>0$ such that for $|k|<k_{p/q}$, the vertical rotation set of an adequate lift of $f_{k,p/q}$ is 
just $\{p/q\}.$

\item when $t$ is irrational, for any $k\neq 0$ the vertical rotation set of the adequate lift of $f_{k,t}$ is a non-degenerate interval which contains $t$ in its interior. 
\end{enumerate}
In other words, it is not easy to build area-preserving twist maps whose vertical rotation sets are reduced to a single irrational number. From Theorem A of \cite{eujul}, such a map needs to have an invariant foliation by Lipschitz graphs over the horizontal coordinate. In particular, all its iterates must satisfy a twist condition. This is precisely what does not hold for $f_{k,t}$, for all non-zero values of $k$.
\medskip


\end{abstract}

\tableofcontents

\section{Introduction}

Rotation sets describe the possible asymptotic displacements of orbits and provide a link between the topology and the dynamics of surface homeomorphisms. For homeomorphisms of the torus homotopic to a Dehn twist, the relevant rotation set is one-dimensional: it records the average vertical displacement of orbits of a lift to the cylinder and is always a nonempty compact interval, see
\cite{non04} and \cite{alunofranks}.
 A natural question is how the degeneration of this interval to a single point constrains the dynamics, and whether this property persists under perturbations.

The affine twist map

$$
\operatorname{Aff}_t(x,y)
=(x+y\bmod 1,\;y+t\bmod 1), 
\qquad t\in\mathbb R,  
$$
provides a simple model for this question. It preserves area and the foliation of the torus by horizontal circles, and its natural lift to \(\mathbb T^1\times\mathbb R\) has vertical rotation set \(\{t\}\). The map induced on the space of leaves is the circle rotation by \(t\). Thus, when \(t\) is rational every leaf is periodic, whereas for irrational \(t\) the orbit of every leaf is dense in the torus. Although the vertical rotation set is a singleton in both cases, we show that its response to the perturbations considered below is very different.

We study the two-parameter family

\begin{eqnarray}
f_{k,t}(x,y)
=\bigl(x+y+k\sin(2\pi x)\bmod 1,\;
y+k\sin(2\pi x)+t\bmod 1\bigr),
\qquad k,t\in\mathbb R. \label{deffkt}
\end{eqnarray}

For \(k=0\) this is the affine twist map, while for \(t=0\) it is the standard map. Every member of the family preserves area and satisfies the twist condition. We choose the cylinder lift \(\widehat f_{k,t}\) whose vertical displacement is \(k\sin(2\pi x)+t\). Since the sine has zero average, the vertical rotation number of Lebesgue measure remains equal to \(t\) for every \(k\). The question is whether all orbits must share this same average vertical displacement. 

Our main theorem establishes a sharp distinction between rational and irrational parameters. If \(t=p/q\), then for all sufficiently small \(|k|\), depending on \(p/q\), the vertical rotation set of \(\widehat f_{k,p/q}\) is still \(\{p/q\}\). Indeed, KAM theory supplies essential invariant curves for the corrected iterate \(\widehat f_{k,p/q}^{\,q}-(0,p)\). These curves and their vertical integer translates bound the vertical motion and force the rotation set to remain a singleton. In contrast, if \(t\) is irrational, then

$$
t\in\operatorname{int}\rho_V(\widehat f_{k,t})
\qquad\text{for every }k\ne0.
$$

This conclusion holds for every irrational \(t\), and for every nonzero perturbation amplitude, however small or large.

It is well-known that existence of interior for rotation sets imply complicated dynamics, like positive topological entropy, and in the $C^{1+\varepsilon}$ world (see \cite{c1eps}), lots of hyperbolic periodic points with rotational horseshoes, a nice description of the so called Region of Instability in the area-preserving case, etc.

The proof of the irrational case uses the rigidity of area-preserving twist maps with a single irrational vertical rotation number proved in \cite{eujul}. There, it was shown that such maps preserve a foliation by Lipschitz graphs over the horizontal coordinate. A consequence is that every positive iterate of the map satisfies the topological twist condition. We show that this necessary condition fails for \(f_{k,t}\) whenever \(k\ne0\): for every such \(k\) and every \(t\), there exist an integer \(n>0\) and a point \(z\in\mathbb R^2\) such that

$$
\frac{\partial}{\partial y}
\bigl(p_1\circ\widetilde f_{k,t}^{\,n}\bigr)(z)<0.
$$

The argument relates the derivatives of the iterates to the leading principal determinants of a symmetric tridiagonal matrix. A variational construction, using an ergodic component of Lebesgue measure, produces a negative value of the associated quadratic form and hence the desired failure of the twist condition. Finally, bounds on vertical deviations for measures whose rotation numbers lie at an endpoint of the rotation interval show that the Lebesgue rotation number \(t\) belongs to its interior.

After introducing the notation and stating the main theorem, we collect the required results on invariant foliations, vertical deviations and the twist condition in Section 2. Section 3 contains the proof, treating the rational case by KAM theory and the irrational case by the variational argument described above.

\subsection{Definitions}

\begin{description}
\item[(1)]  Let $\mathbb{T}^{1}:=\mathbb{R}/\mathbb{Z}$ and $\mathbb{T}^{2}:=%
\mathbb{R}^{2}/\mathbb{Z}^{2}$ be the flat torus. A point in $\mathbb{R}^{2}$
will be denoted as $\widetilde{z}=(\widetilde{x},\widetilde{y})$, in the
cylinder $\mathbb{T}^{1}\times \mathbb{R}$ as $\widehat{z}=(\widehat{x},%
\widehat{y})$, and a point in $\mathbb{T}^{2}$ will be denoted as $z=(x,y)$.
Let $p_{1},p_{2}:\mathbb{R}^{2}\to \mathbb{R}$ denote the canonical
projections, given by $p_{1}(\widetilde{x},\widetilde{y}):=\widetilde{x}$
and $p_{2}(\widetilde{x},\widetilde{y}):=\widetilde{y}$. The projections on
the torus and on the cylinder will be denoted in the same way. Finally, let $%
p:\mathbb{R}^{2}\to \mathbb{T}^{2}$, $\pi :\mathbb{R}^{2}\to \mathbb{T}%
^{1}\times \mathbb{R}$, and $\tau :\mathbb{T}^{1}\times \mathbb{R}\to %
\mathbb{T}^{2}$ be the covering maps defined as $p(\widetilde{x},\widetilde{y%
}):=(\widetilde{x}\bmod1,\widetilde{y}\bmod1)$, $\pi (\widetilde{x},%
\widetilde{y}):=(\widetilde{x}\bmod1,\widetilde{y}),$ and $\tau (\widehat{x},%
\widehat{y}):=(\widehat{x},\widehat{y}\bmod1).$

\item[(2)]  For $\ell \in \mathbb{Z}\setminus \{0\}$, let $\mathrm{Diff}_{%
\mathrm{tw}}^{1}(\mathbb{T}^{2})$ denote the set of $C^{1}$-diffeomorphisms
of $\mathbb{T}^{2}$ that are homotopic to a \emph{Dehn twist}: $(x,y)\mapsto
(x+\ell y\bmod1,y\bmod1)$, and satisfy a \emph{twist condition}: there
exists a constant $k_{tw}>0$ such that, for any lift $\widetilde{f}:%
\mathbb{R}^{2}\to \mathbb{R}^{2}$ of $f$, 
\begin{equation}
\frac{\ell }{|\ell |}\times \frac{\partial p_{1}\circ \widetilde{f}(%
\widetilde{x},\widetilde{y})}{\partial \widetilde{y}}>k_{tw}\quad \text{for
all }(\widetilde{x},\widetilde{y})\in \mathbb{R}^{2}.  \label{equaI02}
\end{equation}
If $\ell >0$, we say that $f$ is a \emph{right twist map}, and when it is
smaller than zero, it is a left one. The inverse of a right twist map is a
left one, and vice-versa.

We also denote by $\mathrm{Diff}_{\mathrm{tw}}^{1}(\mathbb{R}^{2})$
(respectively, $\mathrm{Diff}_{\mathrm{tw}}^{1}(\mathbb{T}^{1}\times %
\mathbb{R})$) the set of diffeomorphisms of the plane (respectively, of the
cylinder) that are lifts of elements of $\mathrm{Diff}_{\mathrm{tw}}^{1}(%
\mathbb{T}^{2})$. If an element of $\mathrm{Diff}_{\mathrm{tw}}^{1}(%
\mathbb{T}^{2})$ is denoted $f$, any lift of $f$ to the vertical cylinder is
denoted as $\widehat{f}$ and to the plane as $\widetilde{f}$.

\item[(3)]  For any $\widehat{f}\in \mathrm{Diff}_{\mathrm{tw}}^{1}(%
\mathbb{T}^{1}\times \mathbb{R})$, we define the \emph{vertical rotation set}
of $\widehat{f}$ as in \cite{alunofranks} or \cite{non04}, 
\begin{equation}
\rho _{V}(\widehat{f}):=\bigcap_{i=1}^{\infty }\overline{\bigcup_{n\geq
i}\left\{ \dfrac{p_{2}\circ \widehat{f}^{n}(\widehat{z})-p_{2}(\widehat{z})}{%
n}:\widehat{z}\in \mathbf{T}^{1}\times \mathbb{R}\right\} }\subset \mathbb{R}%
.  \label{equaI03}
\end{equation}
It was proved in \cite{alunofranks} and \cite{non04} that $\rho _{V}(%
\widehat{f})\neq \emptyset $ is compact and connected, therefore it is
either a singleton or a non-degenerate compact interval. 

\item[(4)]  We denote by $\mathcal{F}_{G}$ a foliation of $\mathbb{T}^{2}$
by simple closed curves, all homotopic to $(1,0)$, each of which is the
graph of a $K$-Lipschitz function of the $x$-coordinate, for some constant $%
K=K(\mathcal{F}_{G})>0$.
\end{description}

\subsection{Main theorem}

For real numbers $k,t,$ let us consider the Standard perturbation of the Affine twist map, 
$f_{k,t}:\mathbb{T}^2\rightarrow \mathbb{T}^2,$ defined in expression (\ref{deffkt}), and the fixed lifts 
of $f_{k,t}$ to the vertical cylinder and to the plane: 
$$
\widehat{f}_{k,t}:\mathbb{T}^{1}\times \mathbb{R}\rightarrow \mathbb{T}^{1}\times \mathbb{R}
\text{ given by }\widehat{f}_{k,t}(\widehat{x},\widehat{y}):=(\widehat{x}+\widehat{y}+k\sin
(2\pi \widehat{x}) \text{ mod 1},\widehat{y}+k\sin (2\pi \widehat{x})+t)
$$
$$
\text{ and }
$$
$$
\widetilde{f}_{k,t}:\mathbb{R}^2\rightarrow \mathbb{R}^2
\text{ given by }\widetilde{f}_{k,t}(\widetilde{x},\widetilde{y}):=(\widetilde{x}+%
\widetilde{y}+k\sin (2\pi \widetilde{x}),\widetilde{y}+k\sin (2\pi 
\widetilde{x})+t) 
$$

Note that $f_{k,t}$ preserves area and belongs to $\mathrm{Diff}_{\mathrm{tw}%
}^{1}(\mathbb{T}^{2})$ for all values of $k$ and $t.$

\begin{theorem}
\label{main} Given $t\in \mathbb{R},$ there are two possibilities:

\begin{enumerate}
\item  if $t$ is a rational $p/q$ for integers $p$ and $q>0,$
then there exists $k_{p/q}>0$ such that for $\left| k\right| <k_{p/q},$ the
vertical rotation set $\rho _{V}(\widehat{f}_{k,p/q})=\{p/q\}$ and the map $%
\left( \widehat{f}_{k,p/q}\right) ^{q}(\bullet )-(0,p)$ has KAM invariant
curves, that is, for certain smooth functions $\phi :\mathbb{T}%
^{1}\rightarrow \mathbb{R},$ 
\[
\left( \left(\widehat{f}_{k,p/q}\right)^q(\bullet )-(0,p)\right) (graph(\phi ))=graph(\phi
).
\]

\item  if $t$ is irrational, then for all $k\neq 0,$ $\rho _{V}(\widehat{f}%
_{k,t})$ is a closed non-degenerate interval containing $t$ in its interior.
\end{enumerate}
\end{theorem}

\begin{remark}
If instead of $\sin (2\pi x),$ we consider any 1-periodic continuous
function $\psi :$ $\mathbb{T}^{1}\rightarrow \mathbb{R}$ satisfying $\int_{%
\mathbb{T}^{1}}\psi (x)dx=0,$ then the following happens: if $\psi $ is
at least $C^{1}$ and non constant, in case of irrational $t,$ we get
the same conclusion as above, with the same proof. And in case $\psi $ is at least $C^{5},$ the
rational case also behaves as above, again from KAM theorem, see \cite{kam} page 680.
\end{remark}

\begin{corollary} As the extremes of the vertical rotation set vary continuously with respect to the parameters 
$k$ and $t,$ for any given  sequence $p_n/q_n$ converging to some irrational number, if $k_n>0$ satisfies 
$$\rho _{V}(\widehat{f}_{k_n,p_n/q_n})=\{p_n/q_n\}, \text{ then } k_n\overset{n\rightarrow \infty}\rightarrow 0.$$ 
\end{corollary}  

In other words, it is not a simple task to produce non trivial examples of
area-preserving twist maps $f:\mathbb{T}^{2}\rightarrow\mathbb{T}^{2}$ whose 
vertical rotation sets are reduced
to a single irrational number $\alpha .$ According to Theorem A of \cite{eujul}, 
in such a case, $f$ preserves a foliation of the form $\mathcal{F}%
_{G},$ and the dynamics induced on the leaves of $\mathcal{F}_{G}$ by $f$ is
conjugate to the irrational rotation by $\alpha $ on the circle.

The fact that in case of irrational $t,$ $interior(\rho _{V}(\widehat{f}%
_{k,t}))$ is not empty for all $k\neq 0$, implies complicated dynamics.
Among other things, the topological entropy $h_{top}(f_{k,t})>0$ and for all
rationals $r/s\in interior(\rho _{V}(\widehat{f}_{k,t})),$ the map $f_{k,t}$
has periodic hyperbolic saddles of vertical rotation number $r/s$ which have
a full mesh. This means that $\left( \widehat{f}_{k,t}\right) ^{s}(\bullet
)-(0,r)$ has a family of periodic hyperbolic saddle points, given by some
point $\widehat{z}_{r/s}\in \mathbb{T}^{1}\times \mathbb{R}$ and all its
integer vertical translates, such that the unstable manifold $W^{u}(\widehat{%
z}_{r/s})$ has a topologically transverse intersection with $W^{s}(\widehat{z%
}_{r/s})+(0,n)$ for all integers $n.$ In particular, this implies that 
\[
\overline{W^{u}(z_{r/s})}=\overline{W^{s}(z_{r/s})}=Instability\text{ }Region%
\text{ }of\text{ }f_{k,t},
\]
where $z_{r/s}=\tau (\widehat{z}_{r/s})$ is $f_{k,t}$-periodic. The $%
Instability$ $Region$ $of$ $f_{k,t}$ is defined as the complement of 
\[
Inessential(f_{k,t}):=\{z\in \mathbb{T}^{2}:z\text{ belongs to some }f_{k,t}%
\text{-periodic open topological disk}\}.
\]
From the existence of the full mesh, $W^{u}(z_{r/s})$ and $W^{s}(z_{r/s})$
contain compact connected arcs, respectively $\lambda ^{u}$ and $\lambda ^{s}
$ both containing $z_{r/s},$ such that each connected component of $\left(
\lambda ^{u}\cup \lambda ^{s}\right) ^{c}$ is an open topological disk, whose lifts to the plane have 
diameter smaller than some $M>0.$ As $f_{k,t}$ preserves area, its a
classical consequence of Poincar\'e's Recurrence Theorem that a connected
component of $Inessential(f_{k,t})$ can not intersect $\lambda ^{u}\cup
\lambda ^{s},$ because such an intersection would create wandering domains$.$
Therefore the diameters of lifts to the plane of all $f_{k,t}$-periodic open topological disks are
bounded from above by $M.$ All these results appeared in \cite{c1eps}. We
mentioned them here in order to put the main theorem into perspective.

As a final comment, this paper answers a question posed to the author by Alejandro Kocsard. 
We thank him for such an interesting problem. At first, the idea was to understand if for some
irrational $t$ and a non-zero $k$, the rotation set of $\widehat{f}_{k,t}$ could be reduced to $t$.
As our main result shows, it can not. 

\section{Preliminaries and Auxiliary Results}

Consider an area-preserving $f\in \mathrm{Diff}_{\mathrm{tw}}^{1}(\mathbb{T}%
^{2}),$ without loss of generality for some integer $\ell >0$ (so the twist
is to the right), and fix lifts $\widehat{f}\in \mathrm{Diff}_{\mathrm{tw}%
}^{1}(\mathbb{T}^{1}\times \mathbb{R})$ of $f,$ and $\widetilde{f}\in 
\mathrm{Diff}_{\mathrm{tw}}^{1}(\mathbb{R}^{2})$ of $\widehat{f}.$ If $\rho
_{V}(\widehat{f})=\{\alpha \}$ for some irrational $\alpha ,$ then the
following results from \cite{eujul} hold:

\begin{theorem}
Under the above hypotheses, $f$ admits an invariant foliation $\mathcal{F}%
_{G}$.
\end{theorem}

As a consequence, all iterates of $f$ satisfy a twist condition:

\begin{corollary}
Under the hypotheses above, for all $n>0$ the iterates $f^{n}$ satisfy the
topological twist condition.
\end{corollary}

\begin{remark}
In this $C^{1}$ case, satisfying the topological twist condition means that 
\[
\frac{\partial }{\partial \widetilde{y}}p_{1}\circ \widetilde{f}^{n}(%
\widetilde{x},\widetilde{y})\geq 0\text{ for all }(\widetilde{x},\widetilde{y%
})\in \mathbb{R}^{2}\text{ and all }n>0.
\]
\end{remark}

So we get the following corollary, which will be at the heart of the proof
of Theorem \ref{main}.

\begin{corollary}
\label{corolimp} Let $\widehat{f}\in \mathrm{Diff}_{\mathrm{tw}}^{1}(%
\mathbb{T}^{1}\times \mathbb{R})$ be an area-preserving twist map, whose
twist is to the right $(\ell >0),$ and fix some $\widetilde{f}:\mathbb{R}%
^{2}\rightarrow \mathbb{R}^{2},$ a lift of $\widehat{f}.$ Suppose that for
some $n>0$ and some $(\widetilde{x},\widetilde{y})\in \mathbb{R}^{2},$ 
\[
\frac{\partial }{\partial \widetilde{y}}p_{1}\circ \widetilde{f}^{n}(%
\widetilde{x},\widetilde{y})<0.
\]
In this case, if $\alpha =\int_{\mathbb{T}^{2}}\left[ p_{2}\circ \widehat{f}%
(\bullet )-p_{2}(\bullet )\right] d$Leb is irrational, then $\alpha \in
interior(\rho _{V}(\widehat{f})).$
\end{corollary}

\begin{proof} 
From the previous results, $\rho _{V}(\widehat{f})$ can not be equal
to $\{\alpha \}.$ So, either $\rho _{V}(\widehat{f})$ is a single rational
point, or it has interior. The condition $\alpha =\int_{\mathbb{T}%
^{2}}\left[ p_{2}\circ \widehat{f}(\bullet )-p_{2}(\bullet )\right] d$Leb is irrational,
means that the vertical rotation number of Lebesgue measure is equal to $%
\alpha .$ As the rotation number of any Borel probability $f$-invariant
measure belongs to $\rho _{V}(\widehat{f}),$ we get that $\rho _{V}(\widehat{%
f})$ must have interior.

So, we are left to show that the vertical rotation number of the Lebesgue
measure does not belong to the boundary of $\rho _{V}(\widehat{f}).$ This
follows from the following result, adapted from Lemma 11 of \cite{bounded}
to our setting.

\begin{lemma}
\label{fabio} Suppose $\widehat{f}\in \mathrm{Diff}_{\mathrm{tw}}^{1}(%
\mathbb{T}^{1}\times \mathbb{R})$ has a vertical rotation set $\rho _{V}(%
\widehat{f})$ with interior. Denoting by $f$ the torus map lifted by $%
\widehat{f},$ let $\mu $ be a Borel probability $f$-invariant measure such
that 
$$
\rho _{V}(\mu )=\int_{\mathbb{T}^{2}}\left[ p_{2}\circ \widehat{f}%
(\bullet )-p_{2}(\bullet )\right] d\mu 
$$
 is a boundary point of $\rho _{V}(%
\widehat{f}).$ Then, if $x^{\prime }\in supp(\mu ),$ for any $\widehat{x}%
^{\prime }\in \tau ^{-1}(x^{\prime })$ and any integer $n>0,$

\begin{eqnarray}
\left| p_{2}\circ \widehat{f}^{n}(\widehat{x}^{\prime })-p_{2}(\widehat{x}%
^{\prime })-n.\rho _{V}(\mu )\right| \leq M_{f},  \label{ffff}
\end{eqnarray}
for some constant $M_{f}>0$ which depends only on $f$ and can be explicitly
computed.
\end{lemma}

\begin{remark}
In our setting, the constant $M_{f}=2+C,$ where $C>0$ comes from
expression (33) of \cite{eujul}.
\end{remark}

From the previous lemma, if $\rho _{V}(\widehat{f})$ has interior, then $%
\rho _{V}(Lebesgue)$ can not belong to the boundary of $\rho _{V}(\widehat{f}%
),$ because there are points in the torus whose vertical rotation numbers
are different from $\rho _{V}(Lebesgue),$ and at these points inequality (%
\ref{ffff}) would not hold for $\mu =Lebesgue.$ So, the corollary is proved.
\end{proof}

We conclude this subsection with a lemma on the twist condition:

\begin{lemma}
\label{twistcondn} Let $\widetilde{f}\in \mathrm{Diff}_{\mathrm{tw}}^{1}(%
\mathbb{R}^{2})$ be such that $\widetilde{f}^{n}$ does not satisfy the topological twist 
condition for some $n>0.$ Then, for all $m\geq n,$ $\widetilde{f}^{m}$ does not satisfy it either.
\end{lemma}

\begin{proof} 
Without loss of generality, assume $\widetilde{f}$ is a right twist map. As $\widetilde{f}^{n}$ does not satisfy the topological twist condition, for some 
$\widetilde{x}_{0},
\widetilde{x}_{1}\in \mathbb{R},$ the intersection $%
\widetilde{f}^{n}(V_{\widetilde{x}_{0}})\cap V_{\widetilde{x}_{1}}$ contains
more than one point (denote $V_{\widetilde{x}}:=\{\widetilde{x}\}\times %
\mathbb{R}).$ Then as the curve $\mathbb{R}\ni t\rightarrow \widetilde{f}%
^{n}(\widetilde{x}_{0},t)$ is a negative curve, there exists a closed
interval $[t_{F},t_{L}]\subset \mathbb{R}$ such that for $t<t_{F},$ $%
\widetilde{f}^{n}(\widetilde{x}_{0},t)$ is to the left of $V_{\widetilde{x}%
_{1}},$ and for $t>t_{L},$ $\widetilde{f}^{n}(\widetilde{x}_{0},t)$ is to
the right of $V_{\widetilde{x}_{1}},$ and moreover, $\widetilde{f}^{n}(%
\widetilde{x}_{0},t_{F})$ is the highest point in $\widetilde{f}^{n}(V_{%
\widetilde{x}_{0}})\cap V_{\widetilde{x}_{1}}$ and $\widetilde{f}^{n}(%
\widetilde{x}_{0},t_{L})$ is the lowest (see Le Calvez \cite
{le1991proprietes}). This means that $[t_{F},t_{L}]\ni t\rightarrow $ $%
\widetilde{f}^{n}(\widetilde{x}_{0},t)$ is a simple arc connecting 2 points
in $V_{\widetilde{x}_{1}},$ going from the highest, to the lowest$.$ So,
there exists $t^{*}\in [t_{F},t_{L}]$ such that 
\[
\frac{\partial }{\partial \widetilde{y}}\widetilde{f}^{n}(\widetilde{x}%
_{0},t^{*})=(0,-a),\text{ for some constant }a>0.
\]
But then, 
\[
D\widetilde{f}\mid _{\widetilde{f}^{n}(\widetilde{x}_{0},t^{*})}.\left( 
\begin{array}{l}
0 \\ 
-a
\end{array}
\right) =\left( 
\begin{array}{l}
\frac{\partial }{\partial \widetilde{y}}\widetilde{f}\left( \widetilde{f}%
^{n}(\widetilde{x}_{0},t^{*})\right) .(-a) \\ 
\ast **
\end{array}
\right) ,
\]
and the first coordinate $\frac{\partial }{\partial \widetilde{y}}\widetilde{%
f}\left( \widetilde{f}^{n}(\widetilde{x}_{0},t^{*})\right) .(-a)<0$ because $%
\widetilde{f}$ is a right twist map. So, $\widetilde{f}^{n+1}$ also does not
satisfy the twist condition. The proof follows by induction.
\end{proof}

\section{Proofs}

First, assume $t=p/q.$ Then, for $k=0,$ 
\[
\left( \widetilde{f}_{0,t}\right) ^{q}(\widetilde{x},%
\widetilde{y})-(0,p)=(\widetilde{x}+q\widetilde{y}+\frac{p(q-1)}{2},\;%
\widetilde{y}).
\]
So, $\left( \widetilde{f}_{0,t}\right) ^{q}(\bullet)-(0,p)$ is a lift to the 
plane of the area-preserving exact integrable
twist map of the vertical cylinder, $\left( \widehat{f}_{0,t}\right) ^{q}(%
\bullet)-(0,p).$ Thus, KAM theorem (see 
\cite{kam} page 680), implies that for a constant $k_{p/q}>0,$ if $\left| k\right|
<k_{p/q},$ then $\left( \widehat{f}_{k,t}\right) ^{q}(\bullet)-(0,p)$ 
has many KAM invariant curves, and so orbits of points in 
$\mathbb{T}^{1}\times \mathbb{R}$ under $\left( \widehat{f}_{k,t}\right) ^{q}(\bullet)-(0,p)$ 
are uniformly bounded. Therefore, 
\[
\rho _{V}(\left( \widehat{f}_{k,t}\right) ^{q}(\bullet)-(0,p))=\{0\}\text{ and so }\rho _{V}(\widehat{f}_{k,t})=\{p/q\}.
\]

Now assume $t$ is irrational and $k\neq 0$ is any real number. From
Corollary \ref{corolimp}, it is enough to show that for some $n>0$ and $(%
\widetilde{x},\widetilde{y})$ in the plane, $\frac{\partial }{\partial 
\widetilde{y}}p_{1}\circ \widetilde{f}_{k,t}^{n}(\widetilde{x},\widetilde{y}%
)<0.$ To see this, note that 
\[
\widetilde{f}_{k,t}(\widetilde{x},\widetilde{y})=(\widetilde{x}+\widetilde{y}%
+k\sin (2\pi \widetilde{x}),\;\widetilde{y}+k\sin (2\pi \widetilde{x})+t)
\]
satisfies: 
\[
\int_{\mathbb{T}^{2}}\left[ p_{2}\circ \widetilde{f}(\bullet )-p_{2}(\bullet
)\right] d\text{Leb}=\int_{\mathbb{T}^{2}}\left[ k\sin (2\pi \widetilde{x}%
)+t\right] d\text{Leb}=t.
\]

Writing, $\widetilde{f}_{k,t}^{n}(\widetilde{x},\widetilde{y})=(p_{1}\circ 
\widetilde{f}_{k,t}^{n}(\widetilde{x},\widetilde{y}),p_{2}\circ \widetilde{f}%
_{k,t}^{n}(\widetilde{x},\widetilde{y})),$ we get that 
\[
\frac{\partial }{\partial \widetilde{y}}p_{1}\circ \widetilde{f}_{k,t}^{n}(%
\widetilde{x},\widetilde{y})=\frac{\partial }{\partial \widetilde{y}}\left(
p_{1}\circ \widetilde{f}_{k,t}^{n-1}(\widetilde{x},\widetilde{y})+p_{2}\circ 
\widetilde{f}_{k,t}^{n-1}(\widetilde{x},\widetilde{y})+k\sin \left( 2\pi
.p_{1}\circ \widetilde{f}_{k,t}^{n-1}(\widetilde{x},\widetilde{y})\right)
\right) =
\]

\[
=\frac{\partial }{\partial \widetilde{y}}p_{1}\circ \widetilde{f}%
_{k,t}^{n-1}(\widetilde{x},\widetilde{y})\left[ 1+2k\pi \cos \left( 2\pi
.p_{1}\circ \widetilde{f}_{k,t}^{n-1}(\widetilde{x},\widetilde{y})\right)
\right] +\frac{\partial }{\partial \widetilde{y}}p_{2}\circ \widetilde{f}%
_{k,t}^{n-1}(\widetilde{x},\widetilde{y}). 
\]

The expression of $\widetilde{f}_{k,t}$ implies that $p_{1}\circ \widetilde{f%
}_{k,t}^{n-1}(\widetilde{x},\widetilde{y})=p_{1}\circ \widetilde{f}%
_{k,t}^{n-2}(\widetilde{x},\widetilde{y})+p_{2}\circ \widetilde{f}%
_{k,t}^{n-1}(\widetilde{x},\widetilde{y})-t.$ So, 
\[
\frac{\partial }{\partial \widetilde{y}}p_{2}\circ \widetilde{f}_{k,t}^{n-1}(%
\widetilde{x},\widetilde{y})=\frac{\partial }{\partial \widetilde{y}}%
p_{1}\circ \widetilde{f}_{k,t}^{n-1}(\widetilde{x},\widetilde{y})-\frac{%
\partial }{\partial \widetilde{y}}p_{1}\circ \widetilde{f}_{k,t}^{n-2}(%
\widetilde{x},\widetilde{y}),
\]
which gives the recurrence we need: 
\[
\frac{\partial }{\partial \widetilde{y}}p_{1}\circ \widetilde{f}_{k,t}^{n}(%
\widetilde{x},\widetilde{y})=\frac{\partial }{\partial \widetilde{y}}%
p_{1}\circ \widetilde{f}_{k,t}^{n-1}(\widetilde{x},\widetilde{y})\left[
2+2k\pi \cos \left( 2\pi .p_{1}\circ \widetilde{f}_{k,t}^{n-1}(\widetilde{x},%
\widetilde{y})\right) \right] -\frac{\partial }{\partial \widetilde{y}}%
p_{1}\circ \widetilde{f}_{k,t}^{n-2}(\widetilde{x},\widetilde{y}).
\]

Denoting $(x_{0},y_{0})=(x,y)=p(\widetilde{x},\widetilde{y})$ and $%
f_{k,t}^{n}(x,y):=(x_{n},y_{n}),$ we get that 
\begin{equation}
y_{n}=y_{n-1}+k\sin (2\pi x_{n-1})+t\text{ mod 1, and }x_{n}=x_{n-1}+y_{n}-t%
\text{ mod 1}  \label{orbit}
\end{equation}

satisfy the following: If we define $a_{0}=0,$ $a_{1}=1$, and for $n\geq 1,$ 
\begin{equation}
a_{n+1}=\left[ 2+2\pi k\cos (2\pi x_{n})\right] a_{n}-a_{n-1},
\label{arecurrence}
\end{equation}
then, 
\[
\frac{\partial }{\partial \widetilde{y}}p_{1}\circ \widetilde{f}_{k,t}^{n}(%
\widetilde{x},\widetilde{y})=a_{n}.
\]

Thus, in order to show that some iterate of $\widetilde{f}_{k,t}$ does not
satisfy a twist condition, we will find $(x_{0},y_{0})\in \mathbb T^{2}$
such that $a_{n}<0$ for some large enough $n>0.$

\begin{lemma}
For every $k\neq0$ and $t\in \mathbb R$, there exist $n>0$ and $(x_{0},y_{0})\in %
\mathbb T^{2}$ such that $a_{n}<0.$
\end{lemma}

\begin{proof} 
As $f_{k,t}$ is conjugate to $f_{-k,t}$ by the map $\varphi
(x,y)=(x+1/2$ mod 1$,y$ mod 1$),$ it suffices to consider $k>0$ and $t\in %
\mathbb R.$
The proof is based on a variational argument for a tridiagonal matrix.
\vskip0.1truecm

\textit{Part 1. The tridiagonal matrix associated with the recurrence.}

Fixed some $(x_{0},y_{0})\in \mathbb T^{2},$ the points $%
(x_{1},y_{1}),(x_{2},y_{2}),...,(x_{N},y_{N}),...$ are its orbit under $f_{k,t}.$

For an integer $N\geq 1$, consider the symmetric $N\times N$ matrix 
\begin{equation}
H_{N}=\begin{pmatrix} 2+2k\pi\cos (2\pi x_1) & -1 & 0 & \cdots &0\\
-1&2+2k\pi\cos (2\pi x_2)&-1&\ddots&\vdots\\ 0&-1&2+2k\pi\cos (2\pi
x_3)&\ddots&0\\ \vdots&\ddots&\ddots&\ddots&-1\\ 0&\cdots&0&-1&2+2k\pi\cos
(2\pi x_N) \end{pmatrix}.  \label{HN}
\end{equation}

Let $D_{N}:=\det H_{N}$, and put $D_{-1}:=0$ and $D_{0}:=1$. Expansion of
the determinant along the last row gives 
\begin{equation}
D_{N}=\left[ 2+2k\pi \cos (2\pi x_{N})\right] D_{N-1}-D_{N-2}.
\label{detrec}
\end{equation}
Since $a_{0}=0,$ $a_{1}=1$ and $a_{n+1}=\left[ 2+2k\pi \cos (2\pi x_{n})\right]
a_{n}-a_{n-1},$ we get that $D_{N}=a_{N+1}.$


\vskip0.1truecm
\textit{Part 2. The search for }$N\geq 1,$ $(x_{0},y_{0})\in \mathbb T^{2}$
and $v\in \mathbb R^{N}$ such that $\langle H_{N}v,v\rangle <0.$

Fixed some $(x_{0},y_{0})\in \mathbb T^{2},$  $N\geq 1$ and $v=(v_{1},\ldots
,v_{N})\in \mathbb R^{N}$, the quadratic form associated to $H_{N}$ is 
\begin{eqnarray}
\langle H_{N}v,v\rangle&=&\sum_{j=1}^{N}\left[ 2+2k\pi \cos (2\pi
x_{j})\right] v_{j}^{2}-2\sum_{j=1}^{N-1}v_{j}v_{j+1} \\
&=&\sum_{j=0}^{N}(v_{j+1}-v_{j})^{2}+2k\pi \sum_{j=1}^{N}\cos (2\pi
x_{j})\,v_{j}^{2},  \label{quadratic}
\end{eqnarray}
where we impose that $v_{0}=v_{N+1}=0$.

We shall show that this quadratic form is negative for a suitable orbit and
a suitable vector $v.$

As $f_{k,t}$ preserves Lebesgue measure, 
\begin{eqnarray}
\int_{\mathbb T^{2}}\cos (2\pi x)d\text{Leb}=0\text{ and }\int_{\mathbb %
T^{2}}\cos ^{2}(2\pi x)d\text{Leb}>0,\text{ we get that }  \label{valori}
\end{eqnarray}
for some $f_{k,t}$-invariant ergodic measure $\mu $ in the ergodic
decomposition of Lebesgue, one of the possibilities below must hold:

\begin{enumerate}
\item  $\int_{\mathbb T^{2}}\cos (2\pi x)d\mu <0;$

\item  $\int_{\mathbb T^{2}}\cos (2\pi x)d\mu =0$ and $\int_{\mathbb T^{2}}\cos
^{2}(2\pi x)d\mu >0;$
\end{enumerate}

To see this, let $\eta$ denote the probability measure on the space 
of $f_{k,t}$-invariant ergodic Borel probability measures representing the ergodic decomposition 
of Lebesgue measure. If case 1 does not occur, then

$$
\int_{\mathbb T^2}\cos(2\pi x)d\nu\geq 0
$$
for $\eta$-almost every $\nu$. By (\ref{valori}) and the ergodic decomposition formula,
$$
0=\int_{\mathbb T^2}\cos(2\pi x)\,d\mathrm{Leb}
=\int\left(\int_{\mathbb T^2}\cos(2\pi x)\,d\nu\right)d\eta(\nu).
$$

Thus $\int_{\mathbb T^2}\cos(2\pi x)d\nu=0$ for $\eta$-almost every $\nu$. Moreover,

$$
0<\int_{\mathbb T^2}\cos^2(2\pi x)d\mathrm{Leb}
=\int\left(\int_{\mathbb T^2}\cos^2(2\pi x)d\nu\right)d\eta(\nu),
$$
so the set of ergodic measures $\nu$ for which 
$\int_{\mathbb T^2}\cos^2(2\pi x)d\nu>0$ has positive $\eta$-measure. 
Consequently, there exists an ergodic measure $\mu$ satisfying both conditions in case 2.


Let us deal with both possibilities.

In the first, picking some $\mu $-generic point $(x_{0},y_{0})$, Birkhoff's
ergodic theorem gives 
\[
\frac{1}{N}\sum_{j=1}^{N}\cos (2\pi x_{j})\stackrel{N\rightarrow \infty }{%
\longrightarrow }\int_{\mathbb T^{2}}\cos (2\pi x)\,d\mu <0.
\]
So, take $v_{j}=1$ for $1\leq j\leq N$ and $v_{0}=v_{N+1}=0$. From
expression (\ref{quadratic}), 
\[
\langle H_{N}v,v\rangle =2+2k\pi \sum_{j=1}^{N}\cos (2\pi x_{j}).
\]
From the above, $\sum_{j=1}^{N}\cos (2\pi x_{j})$ goes to $-\infty$ as 
$N\rightarrow \infty .$ So, for sufficiently large $%
N>0,$ 
\[
\langle H_{N}v,v\rangle <0.
\]

The second possibility is more complicated. Pick a $\mu $-generic point $%
(x_{0},y_{0}).$ Then Birkhoff's ergodic theorem gives 
\[
\frac{1}{N}\sum_{j=1}^{N}\cos (2\pi x_{j})\stackrel{N\rightarrow \infty }{%
\longrightarrow }\int_{\mathbb T^{2}}\cos (2\pi x)\,d\mu =0\text{ and }\frac{%
1}{N}\sum_{j=1}^{N}\cos ^{2}(2\pi x_{j})\stackrel{N\rightarrow \infty }{%
\longrightarrow }\int_{\mathbb T^{2}}\cos ^{2}(2\pi x)\,d\mu :=M_{\mu }>0%
\text{ }.
\]

Remember that 
\[
\langle H_{N}v,v\rangle =\sum_{j=0}^{N}(v_{j+1}-v_{j})^{2}+2k\pi
\sum_{j=1}^{N}\cos (2\pi x_{j})\,v_{j}^{2},\text{ where we impose that }%
v_{0}=v_{N+1}=0.\text{ } 
\]
Let $N>0$ be an integer and let $0<\varepsilon <1,$ to be chosen later. Now
define $v_0=v_{N+1}=0$ and for $1\leq j\leq N:$


$$v_{j}:=\left( 1-\varepsilon
\cos (2\pi x_{j})\right) $$



\[
\text{Then } \sum_{j=0}^{N}(v_{j+1}-v_{j})^{2}<4+\sum_{j=1}^{N-1}(v_{j+1}-v_{j})^{2}+4<
\]

\[
<8+\sum_{j=1}^{N-1}4\varepsilon^{2}<8+4\varepsilon^{2}N.
\]

And for the second term, the following holds:

\[
v_{j}^{2}=(1-\varepsilon \cos (2\pi x_{j}))^{2}=1-2\varepsilon \cos (2\pi
x_{j})+\varepsilon ^{2}\cos ^{2}(2\pi x_{j}). 
\]
So, 
\[
\cos (2\pi x_{j})v_{j}^{2}=\cos (2\pi x_{j})-2\varepsilon \cos ^{2}(2\pi
x_{j})+\varepsilon ^{2}\cos ^{3}(2\pi x_{j}). 
\]


Consequently, 
{\small
\begin{eqnarray}
\sum_{j=1}^{N}\cos (2\pi x_{j})v_{j}^{2}\leq\sum_{j=1}^{N}\left[ \cos (2\pi x_{j})-2\varepsilon\cos
^{2}(2\pi x_{j})+\varepsilon ^{2}\cos ^{2}(2\pi x_{j})
\right].  \label{primestim}
\end{eqnarray}
}
From the choice of the measure $\mu$, for every $\theta >0,$ there exists $N_{0}=N_{0}(\theta )>0$
sufficiently large such that, for $N\geq N_{0},$ the
following estimates hold:

\[
\frac{1}{N} \left|\sum_{j=1}^{N}\cos (2\pi x_{j})\right|  <\theta M_{\mu } 
\]

\[
\frac{1}{N}\sum_{j=1}^{N}\cos ^{2}(2\pi x_{j})>\frac{9}{10}M_{\mu } 
\]

So, the estimate in expression (\ref{primestim}) gives: 
\[
\sum_{j=1}^{N}\cos (2\pi x_{j})v_{j}^{2}<\theta NM_{\mu }-\varepsilon \frac{%
9}{5}NM_{\mu }+\varepsilon ^{2}NM_{\mu }= 
\]

\[
=NM_{\mu }\left( \theta -\frac{9}{5}\varepsilon +\varepsilon ^{2}\right) . 
\]
And putting both estimates together, we obtain:

\[
\sum_{j=0}^{N}(v_{j+1}-v_{j})^{2}+2k\pi \sum_{j=1}^{N}\cos (2\pi
x_{j})\,v_{j}^{2}<
8+2k\pi NM_{\mu }\left( \theta -\frac{9}{5}\varepsilon +\varepsilon
^{2}\left( 1+\frac{2}{M_{\mu }k\pi}\right) \right). 
\]
 
As the vertex coordinate of the above polynomial is 
$\frac{9}{10\left( 1+\frac{2}{M_{\mu }k\pi}\right)}$, which belongs to $(0,1),$ 
choosing 
\[
0<\theta <\frac{81}{100\left( 1+\frac{2}{M_{\mu }k\pi}\right)}, 
\]
we get 
an interval $(\varepsilon _{1},\varepsilon _{2})\subset (0,1)$ such
that for any $\varepsilon \in (\varepsilon _{1},\varepsilon _{2}),$ the polynomial 
\[
\theta -\frac{9}{5}\varepsilon +\varepsilon ^{2}\left( 1+\frac{2}{M_{\mu }k\pi}\right) <0.
\]
Therefore, for the above choice of $\theta,$ large enough $N>0$ and 
$\varepsilon \in (\varepsilon _{1},\varepsilon _{2}),$ in this second possibility we also obtain that 
$$
\langle H_{N}v,v\rangle =\sum_{j=0}^{N}(v_{j+1}-v_{j})^{2}+2k\pi
\sum_{j=1}^{N}\cos (2\pi x_{j})\,v_{j}^{2}<0.
$$ 
\vskip0.1truecm

\textit{Part 3. A negative }$D_{j}.$

In Part 2 above, we found $(x_{0},y_{0})\in \mathbb T^{2},$ $N\geq 1$ and $%
v\in \mathbb R^{N}$ such that $\langle H_{N}v,v\rangle <0.$ And if all $%
D_{1},\ldots ,D_{N}$ were positive, Sylvester's criterion would imply that $%
H_{N}$ is positive definite, a contradiction. Thus for some $1\leq j\leq N,$
the determinant $D_{j}\le 0$. 
Let $1\leq j^* \leq N$ be the first index with $D_{j^*}\le 0$. If $D_{j^*}<0$, we
are done. And if $D_{j^*}=0$, then 
\[
D_{j^*+1}=(2+2k\pi\cos (2\pi x_{j^*+1}))D_{j^*}-D_{j^*-1}=-D_{j^*-1}<0.
\]
Hence the lemma is proved, because for all integers $i\geq 0,$ $a_{i+1}=D_{i}.$
\end{proof}

And the main theorem is also proved.

\end{document}